\documentclass[a4paper,11pt,reqno]{amsart}
\usepackage{amsmath,amssymb,amscd,mathrsfs,epic,wasysym,latexsym,tikz,mathrsfs,cite,hyperref,enumerate}
\usepackage{amsthm}
\usepackage{enumitem}
\usepackage{pb-diagram}
\usepackage{tikz-cd}
\usepackage[all]{xy}
\usepackage{mathdots}
\usepackage[new]{old-arrows}
\usepackage[capitalise]{cleveref}

\numberwithin{equation}{section}

\newtheorem{thm}[equation]{Theorem}
\crefname{thm}{Theorem}{Theorems}
\Crefname{thm}{Theorem}{Theorems}
\newtheorem{lemma}[equation]{Lemma}
\crefname{lemma}{Lemma}{Lemmas}
\Crefname{lemma}{Lemma}{Lemmas}
\newtheorem{cor}[equation]{Corollary}
\crefname{cor}{Corollary}{Corollaries}
\Crefname{cor}{Corollary}{Corollaries}

\crefname{defi}{Definition}{Definitions}
\Crefname{defi}{Definition}{Definitions}

\crefname{rem}{Remark}{Remarks}
\Crefname{rem}{Remark}{Remarks}
\newtheorem{MainTheorem}{Theorem}

\crefname{MainTheorem}{Theorem}{Theorems}
\Crefname{MainTheorem}{Theorem}{Theorems}

\newtheorem*{conj}{Conjecture}

\newcommand{\soc}{\mathrm{soc}}

\newcommand{\Hom}{\mathrm{Hom}}
\newcommand{\End}{\mathrm{End}}
\newcommand{\Ext}{\mathrm{Ext}}

\newcommand{\cont}{\mathrm{cont}}

\newcommand{\res}{\mathrm{res}}

\newcommand{\wt}{\mathrm{wt}}
\newcommand{\quot}{\mathrm{quot}}
\newcommand{\core}{\mathrm{cor}}

\newcommand{\s}{{\sf S}}

\def\mod#1{#1\!\operatorname{-mod}}

\newcommand{\Z}{\mathbb{Z}}

\newcommand{\F}{\mathbb{F}}

\renewcommand{\epsilon}{\varepsilon}
\renewcommand{\phi}{\varphi}
\newcommand{\xymat}{\xymatrix@R=6pt@C=10pt}

\newcommand{\la}{\lambda}
\newcommand{\be}{\beta}
\newcommand{\al}{\alpha}
\newcommand{\eps}{\varepsilon}
\renewcommand{\phi}{\varphi}

\newcommand{\de}{\delta}
\newcommand{\Ga}{\Gamma}
\newcommand{\Om}{\Omega}

\newcommand{\da}{{\downarrow}}
\newcommand{\ua}{{\uparrow}}

\def\Par{{\mathscr {P}}}
\def\Parreg{{\mathscr P}^{\text{\rm $p$-reg}}}
\def\Parres{{\mathscr P}^{\text{\rm $p$-res}}}

\newcommand{\te}{{\tilde e}}
\newcommand{\tf}{{\tilde f}}
\newcommand{\he}{{\hat e}}
\newcommand{\hf}{{\hat f}}

\newcommand{\mods}{\textrm{-mod}}

\newcommand{\beps}{\eps'}
\newcommand{\bphi}{\phi'}

\newcommand{\Rem}{{\operatorname{Rem}}}
\newcommand{\Add}{{\operatorname{Add}}}

\def\bA{\text{\boldmath$A$}}
\def\bB{\text{\boldmath$B$}}

\begin{document}

\title[Self-extensions over symmetric groups]{{\bf On self-extensions of irreducible modules over symmetric groups, II}}

\author{\sc Lucia Morotti}
\address
{Department of Mathematics\\ University of York\\ York YO10 5DD, U.K.} 
\email{lucia.morotti@uni-duesseldorf.de}

\thanks{The author thanks Haralampos Geranios for conversations on these results. The author was supported by the Royal Society grant URF$\backslash$R$\backslash$221047.}

\begin{abstract}
It is conjectured that irreducible representations of symmetric groups have no non-trivial self-extension over fields of odd characteristic. We improve on partial results showing evidence of this conjecture. 
\end{abstract}

\maketitle

\section{Introduction}

Let $\s_n$ be the symmetric group and $\F$ be an algebraically closed field of characteristic $p\geq 0$. It is well known, see for example \cite{JamesBook}, that the irreducible representations of $\F\s_n$ are labeled by $p$-regular partitions of $n$. We let $\Parres_n$ denote the set of $p$-regular partitions of $n$ and for $\la\in\Parres_n$ we let $D^\la$ denote the corresponding irreducible $\F\s_n$-representation.

It is conjectured that

\begin{conj}
If $p\geq 3$ then, for every $\la\in\Parres_n$, $D^\la$ has no non-trivial self-extension , that is $\Ext^1(D^\la,D^\la)=0$.
\end{conj}

This conjecture has been proved in specific cases. Before stating the known cases we need some notation. Let $\Par_n$ denote the set of all partitions of $n$. For $\la$ in $\Par_n$ we let $h(\la)$ denote the number of parts of $\la$ and $S^\la$ be the Specht module indexed by $\la$. If $\la\in\Parres_n$ and $i\in I$ we let $\eps_i(\la)$ resp. $\phi_i(\la)$ denote the number of normal resp. conormal nodes of $\la$ of residue $i$ (see \S\ref{subsec:tran.fun} or \cite[\S11.1]{K3} for definitions). For $0\leq a\leq\eps_i(\la)$ resp. $0\leq b\leq\phi_i(\la)$ we let $e_i^{(a)}D^\la$ resp. $f_i^{(b)}D^\la$ be the divided power functors describing restrictions resp. inductions to smaller resp. larger symmetric groups and let $\tilde e_i^a\la$ resp. $\tilde f_i^b\la$ to be the partition obtained from $\la$ by removing the $a$ bottom $i$-normal nodes resp. adding the $b$ top $i$-conormal nodes of $\la$ (all these definitions are also given in \S\ref{subsec:tran.fun} or \cite[\S11.1]{K3}).

Let $\la\in\Parres_p(n)$. In \cite[Theorem 2.10]{KS} it was proved that if $h(\la)\leq p-1$ then $D^\la$ has no non-trivial self-extension. This result was extended to $h(\la)\leq p+2$ in \cite[Theorem C]{GKM}. In \cite[Theorem 3.3(c)]{KN} it was similarly proved that $D^\la$ has no non-trivial self-extension if $p\geq 5$ $D^\la\cong S^\nu$ is an irreducible Specht module (for some partition $\nu\in\Par_n$ possibly different from $\la$). This second result was extended to $p\geq 3$ and improved to $D^{\tilde e_i^{\eps_i(\la)}}$ being an irreducible Specht module in \cite[Theorem D]{GKM}. Further in \cite[Theorems A, B]{GKM} the conjecture was proved for all irreducible modules in certain blocks, namely either RoCK blocks or blocks of weight $\leq 7$.

In the present paper we extend the known cases to include one further family. This family can be viewed as a generalisation of both irreducible Specht modules as well as of modules in RoCK blocks. Before stating this result, we need some more notation. We assume knowledge about the abacus (see \S\ref{subsec:comb} or \cite[\S2.7]{JK}). Let $\la$ be a partition and fix an abacus configuration $\Ga(\la)$ such that no runner of $\Ga(\la)$ is empty. Let $\rho$ be the $p$-core of $\la$ and $\Ga(\rho)$ be the abacus configuration of $\rho$ obtained by moving all beads of $\Ga(\la)$ as high as possible in their runners. For every runner $i$, let $r_i$ be the number of beads on runner $i$ of $\Ga(\la)$ or $\Ga(\rho)$, $m_i$ to be the largest bead on runner $i$ of $\Ga(\rho)$ and $\la^{(i)}$ be the partition on runner $i$ of $\Ga(\la)$. We say that $\la$ is $p$-separated if for every pair $(i,j)$ of runners with $m_i>m_j$ we have that
\[h(\la^{(i)})+\la^{(j)}_1\leq\left\{\begin{array}{ll}
r_i-r_j,&i<j,\\
r_i-r_j+1,&i>j.
\end{array}\right.\]
The above condition is equivalent to saying that, in $\Ga(\la)$, the first gap on runner $i$ comes after the last bead on runner $j$ whenever $m_i>m_j$ (the corresponding positions could though be on the same row of the abacus if $i>j$).

Using this notation we are now ready to state Theorem \ref{thma}.

\begin{MainTheorem}\label{thma}
Let $p\geq 3$, $\al\in\Par_n$ and $\la\in\Parreg_n$. If $\al$ is $p$-separated and $[S^\al:D^\la]>0$ then $\Ext^1(D^\la,D^\la)=0$.
\end{MainTheorem}

Though the conjecture is not yet proved in general, in \cite[Theorem E, Lemma 4.8]{GKM} some reduction results are obtained, limiting minimal partitions $\la$ such that $D^\la$ has non-trivial self-extensions. The next result is an improved version of \cite[Theorem E]{GKM}. In the theorem, for $\la$ a partition, $E_1,\ldots,E_k$ removable nodes of $\la$ and $F_1,\ldots,F_h$ addable nodes of $\la$, we let $\la^{F_1,\ldots,F_h}_{E_1,\ldots,E_k}$ be the partition obtained from $\la$ by removing the nodes $E_1,\ldots,E_k$ and adding the nodes $F_1,\ldots,F_h$.

\begin{MainTheorem}\label{T081223_2}\label{thmc}
Let $\la\in\Par_p(n)$ and $i\in I$. Assume that $V\cong D^\la|D^\la\subseteq f_i^{(\eps_i(\la))}D^{\tilde e_i^{\eps_i(\la)}\la}$ and let $b$ minimal with $V\subseteq f_i^{(b)}D^{\tilde e_i^{b}\la}$. Then $1\leq b\leq\min\{\eps_i(\la),\phi_i(\la)\}$.

Further if $B_1,\ldots,B_{\eps_i(\la)}$ are the $i$-normal nodes of $\la$ counted from the bottom and $C_1,\ldots,C_{\phi_i(\la)}$ are the $i$-conormal nodes of $\la$ counted from the top then $\la_{B_1,\ldots,B_b}^{C_j}$ and $\la_{B_j}^{C_1,\ldots,C_b}$ are $p$-singular for every $1\leq j\leq b$.
\end{MainTheorem}

The final main result of this paper also gives conditions on minimal partitions $\la$ such that $D^\la$ has non-trivial self-extensions.

\begin{MainTheorem}\label{thmb}
Let $p\geq 3$ and $\la\in\Parreg_n$. If $\la$ is minimal such that $D^\la$ has non-trivial self-extensions, then all normal nodes of $\la$ are of the same residue.
\end{MainTheorem}

Theorem \ref{thma} will be proved in Section \ref{separated}, while Theorems \ref{thmc} and \ref{thmb} in Section \ref{sreductions}.

\section{Background}

\subsection{Submodule structure}

In Section \ref{sreductions} we will be studying filtrations of certain modules. In this view we introduce the following notation. If $V,W_1,\ldots,W_h$ are  $\F G\mods$ for some group $G$, we write $V\sim W_1|\ldots|W_h$ if $V$ has a filtration $0=V_0\subseteq V_1\subseteq\ldots\subseteq V_H=V$ with $V_i/V_{i-1}\cong W_i$ for each $1\leq i\leq h$. If $D_1,\ldots,D_h$ are irreducible $\F G\mods$ and $V\sim D_1|\ldots|D_h$ is uniserial, we also write  $V\cong D_1|\ldots|D_h$.

\subsection{Partitions and abaci}\label{subsec:comb}

We will now introduce multiple combinatorial notations that will be used in this paper.

Let $\la$ be a partition. The transpose partition is denoted $\la'$. We denote $|\la|:=\la_1+\la_2+\ldots$. We further denote $h(\la)$ the height of $\la$, that is the number of non-zero parts of $\la$. For a given $n\in\Z_{>0}$, we denote $\Par_n$ to be the set of all partitions of $n$. We denote by $\varnothing$ the trivial partition of $0$, thus $\Par_0=\{\varnothing\}$.  We denote by $\unrhd$ the dominance order on partitions.

A partition $\la$ is called {\em $p$-regular} if none of its parts is repeated $p$ or more times. We denote the sets of $p$-regular partitions of $n$ by $\Parreg_n$. A partition $\la$ that is not $p$-regular is  called \emph{$p$-singular}. 

The elements of $\Z_{>0}\times \Z_{>0}$ are called {\em nodes}. We identify a partition $\la$ with its  \emph{Young diagram}, that is its set of nodes $\{(k,l)\in \Z_{>0}\times \Z_{>0}\mid l \leq \la_k\}$. 

We set $I:=\Z/p\Z$, which we usually identify with $\{0,1,\dots,p-1\}$.
Let $A=(k,l)$ be a node. The \emph{residue} of $A$ is $\res A:=l-k\pmod{p}\in I$. 

The {\em residue content} of a partition $\la\in\Par_n$ is $\cont(\la):=(a_i)_{i\in I}\in\Z^I$, where, for every $i\in I$, $a_i=a_i^\la$ is the number of nodes of $\la$ of content $i$.
 
If $\la\in \Par_n$, then a node $A\in \la$ (resp $B\not\in \la$) is called \emph{removable} (resp. {\em addable}) for $\la$ if  
$\la_A:=\la\setminus\{A\}$ (resp. $\la^B:=\la\cup\{B\}$) is the Young diagram of a partition. Fix $i\in I$. 
A removable (resp. addable) node is called {\em $i$-removable} (resp. {\em $i$-addable}) if it has residue $i$. 
More generally, if $A_1,\ldots, A_k$ are removable (resp. addable) nodes of $\lambda$, we denote $\la_{A_1,\dots,A_k}:=\la\setminus\{A_1,\dots,A_k\}$ (resp. $\la^{A_1,\dots,A_k}:=\la\cup\{A_1,\dots,A_k\}$). 

A beta-set $\Ga$ is a finite subset of $\Z_{\geq 0}$. If $\Ga=\{a_1,\ldots,a_h\}$ with $a_1>\ldots>a_h$, then we say that $\Ga$ is a beta-set of the partition
\[\la(\Ga):=(a_1,a_2-1,\ldots,a_h-h+1).\]
For $\la$ a partition we denote by $\Ga(\la)$ any beta-set of $\la$. We will identify beta-sets with their abacus displays. We will use notation as in \cite[\S2.7]{JK}. Positions on the abacus are labeled with non-negative integers, so that for $i\in I$, the positions $\{i+pa\mid a\in\Z_{\geq 0}\}$ form the runner $i$ of the abacus. The abacus is displayed a set of $p$ columns, each corresponding to a runner, starting with runner $0$ on the left. On each runner, entries increase when moving down rows. A beta-set $\Ga$ can be pictured on the abacus by adding a bead on each entry corresponding to an element of $\Ga$.

A position $k>0$ in $\Ga(\la)$ is {\em removable} (resp. {\em addable}) if it is occupied (resp. unoccupied) and position $k-1$ is unoccupied (resp. occupied). Let $\Ga(\la)=\{a_1,\ldots,a_h\}$ with $a_1>\ldots>a_h$. If $k=a_i$ for some $i$, then $k$ is removable for $\Ga(\la)$ if and only if the node $A_i:=(i,\la_i)$ is removable for $\la$ and in this case $(\Ga(\la)\setminus\{k\})\cup\{k-1\}=\Ga(\la_{A_i})$. Similarly, if $k=a_i+1$ for some $i$, then $k$ is addable for $\Ga(\la)$ if and only if the node $B_i:=(i,\la_i+1)$ is addable for $\la$ and in this case $(\Ga(\la)\setminus\{k-1\})\cup\{k\}=\Ga(\la^{B_i})$. Removable or addable nodes have the same residue if and only if the corresponding removable or addable positions are on the same runner. We denote by $i_j(\Ga)$ the residue of the removable/addable nodes of $\la$ corresponding to removable/addable positions on runner $j$. 

Recall that, for $\la$ a partition, $\Ga(\la)$ is not unique and depends on the number of beads chosen. We will always assume that $0\in\Ga(\la)$. If arguments require to modify $\la$ (or $\Ga(\la)$) by (recursively) removing/adding nodes, we will also assume that $\Ga(\la)$ is chosen so that $0$ is occupied at each corresponding step. This can be easily achieved by adding multiple full rows of beads at the top of the abacus.

We will also need the notions of \emph{core}, \emph{quotient} and \emph{weight} of a partition. The notions of core and weight agree with those used in \cite[\S2.7]{JK}. The notion of weight is however slightly different. Again, let $\la$ be a partition and $\Ga(\la)$ be an abacus configuration of $\la$.

Let $\Lambda$ be the abacus configuration obtained by moving all beads in $\Ga(\la)$ as high as possible within each runner without overstacking beads. The $p$-core (or simply core if $p$ is understood) is the partition $\core(\la)$ with $\Lambda=\Ga(\core(\la))$.

Next we define the $p$-quotient (or simply quotient). Let $0\leq i<p$ and consider runner $i$ of $\Ga(\la)$. Through the bijection $\{i+pc|c\in\Z_{\geq 0}\}\to\Z_{\geq 0}$ given by $i+pc\to c$, we can view runner $i$ as a beta-set of a certain partition $\la^{(i)}$. The quotient of $\la$ is then given by the tuple of partitions $\quot(\la)=(\la^{(0)},\ldots,\la^{(p-1)})$. Note that unlike in \cite[\S2.7]{JK} we make no assumption on the number of beads in $\Ga(\la)$ being divisible by $p$. So the quotient of $\la$ is not well defined. It is however well-defined once a specific abacus configuration is chosen.

The $p$-weight of $\la$ (or weight) is then given by $\wt(\la)=\sum_{i\in I}|\la^{(i)}|$.

Let $\rho\in\Par_r$ be a core, $d\in\Z_{\geq 0}$, and set $n:=r+pd$. We denote by $\Par_{\rho,d}$ the set of all partitions of $n$ with core $\rho$ and weight $d$ and by $\Parres_{\rho,d}:=\Par_{\rho,d}\cap\Parres_n$ the corresponding set of $p$-regular partitions.

\subsection{Representations of symmetric groups}\label{SSSG}

It is well known that irreducible representations of symmetric groups are labeled by partitions in characteristic $0$ and by $p$-regular partitions in characteristic $p$, see for example \cite{JamesBook}. For $\la\in\Par_n$ we thus denote by $S^\la$ the Specht module indexed by $\la$, see \cite[\S4]{JamesBook}. In characteristic $0$ these modules form the set of irreducible representations of $\s_n$, but we will identify the Specht modules $S^\la$ also with their reductions modulo $p$. For $\mu\in\Parreg_n$ we let $D^\mu$ denote the corresponding irreducible representation of $\s_n$ in characteristic $p$, see \cite[\S11]{JamesBook}. By \cite[Theorem 11.5]{JamesBook}, $D^\mu$ is the head of $S^\mu$ whenever $\la$ is $p$-regular.

Let $\rho\in\Par_r$ be a core, $d\in\Z_{\geq 0}$, and $n=r+dp$. We denote by $B_{\rho,d}$ the block of the symmetric group algebra $\F\s_{n}$ corresponding to $\rho$, cf. \cite[6.1.21]{JK}. 
The irreducible  $B_{\rho,d}$-modules are $\{D^\mu\mid \mu\in \Parreg_{\rho,d}\}$, cf. \cite[7.1.13, 7.2.13]{JK}. Further $S^\la$ also belong to the block $B_{\rho,d}$ for all $\la\in\Par_{\rho,d}$. 

By \cite[Theorem 2.7.41]{JK}, it follows that $D^\la$ and $D^\mu$ are in the same block if and only if $\cont(\la)=\cont(\mu)$. For $(a_i)\in\Z^I$ we will thus write $B_{(a_i)}$ for the block corresponding to partitions with content $(a_i)$, if such a block exist. If no such block exist we define $B_{(a_i)}:=0$. We also define $b_{(a_i)}$ to be the corresponding block idempotent (or $0$ if $B_{(a_i)}=0$).

With this notation $\F\s_n=\oplus_{(a_i)\in\Z^I}B_{(a_i)}=\oplus_{(a_i)\in\Z^I}b_{(a_i)}\F\s_n$.


\subsection{Translation functors}
\label{subsec:tran.fun}

Let $\la$ be a partition. For a residue $i\in I$, we denote
\begin{align*}
 \beps_i(\la):=\sharp\{i\text{{-removable nodes of}}\ \la\}\quad\text{and}\quad
 \bphi_i(\la):=\sharp \{i\text{{-addable nodes of}}\ \la\}.
\end{align*}

We will now introduce normal and conormal nodes, using the notation from \cite[\S11.1]{K3}. Label the $i$-addable
nodes of $\la$ with $+$ and the $i$-removable nodes of $\la$ with $-$. The {\em $i$-signature} of 
$\la$ is the sequence of pluses and minuses obtained by going along the 
rim of the Young diagram from bottom left to top right and reading off
all the signs.
The {\em reduced $i$-signature} of $\la$ is obtained 
from the $i$-signature
by successively deleting all neighboring 
pairs of the form $-+$. 
The nodes corresponding to  $-$'s (resp. $+$'s) in the reduced $i$-signature are
called {\em $i$-normal} (resp. {\em $i$-conormal}) for $\la$.
The leftmost $i$-normal (resp. rightmost $i$-conormal) node is called {\em $i$-good} (resp. {\em $i$-cogood}) for $\la$. 
We write
\begin{align*}
\eps_i(\la):=\sharp\{i\text{{-normal nodes of}}\ \la\}\quad\text{and}\quad
 \phi_i(\la):=\sharp \{i\text{{-conormal nodes of}}\ \la\}.
\end{align*}
 
Let $\la\in \Parreg_n$ and $i\in I$. Let $A_1,A_2,\ldots,A_{\eps_i(\la)}$ (resp. $B_1,B_2,\ldots,B_{\phi_i(\la)}$) be the $i$-normal (resp. $i$-conormal) nodes for $\la$, labelled from bottom to top (resp. from  top to bottom). We define 
\begin{equation*}
\te^r_i\la:=\la_{A_1,\ldots,A_r},\quad \tf^r_i\la:=\la^{B_1,\ldots,B_r}.
\end{equation*}
We interpret $\te^r_i\la$ (resp. $\tf^r_i\la$) as $0$ if $r>\eps_i(\la)$ (resp. $r>\phi_i(\la)$). It can be checked combinatorially that the partitions $\te^r_i\la$ and $\tf^r_i\la$ are $p$-regular if $r\leq\eps_i(\la)$ resp. $r\leq\phi_i(\la)$. Moreover $\te_i^r\tf_i^r\la=\la$ (resp. $\tf_i^r\te_i^r\la=\la$) for $r\leq \phi_i(\la)$ (resp. $r\leq \eps_i(\la)$).

Similarly, if $C_1,C_2,\ldots,C_{\beps_i(\la)}$ (resp. $D_1,D_2,\ldots,D_{\bphi_i(\la)}$) are the $i$-removable (resp. $i$-addable) nodes for $\la$, then we define
\begin{equation*}
\he^{\beps_i(\la)}_i\la:=\la_{A_1,\ldots,A_{\beps_i(\la)}},\quad \hf^{\bphi_i(\la)}_i\la:=\la^{B_1,\ldots,B_{\bphi_i(\la)}}.
\end{equation*} 

We now define the {\em $i$-induction} and {\em $i$-restriction}  (translation) functors. Further details can be found in \cite[\S11.2]{K3}. 
Let $j\in I$, $r\in\Z_{\geq 0}$, $(a_i)\in\Z^I$ correspond to a block of $\F\s_n$ and  $V$ be a module over the block $B_{(a_i)}$. Viewing $V$ to a $\F\s_n$-module, define
\begin{align*}
e_i^{(r)}V&:= b_{(a_i)-r\de_j}(V \da^{\s_n}_{\s_{n-r}\times\s_r})^{\s_r}\in\mod{B_{(a_i)-r\de_j}},\\  
f_i^{(r)}V&:= b_{(a_i)+r\de_j}(V\boxtimes\F{\s_r})\ua_{\s_n\times\s_r}^ {\s_{n+r}}\in\mod{B_{(a_i)+r\de_j}},
\end{align*}
with $\de_j=(0,\ldots,0,1,0,\ldots,0)$, with the coefficient $1$ being the coefficient corresponding to $j\in I$.

The definitions of $e_i^{(r)}V$ and $f_i^{(r)}V$ can be extended to any $\F\s_n$-module $V$ additively, obtain functors
$e_i^{(r)}: \F\s_n\mods \to \F\s_{n-r}\mods$ and $f_i^{(r)}: \F\s_n\mods\to \F\s_{n+r}\mods$. 
We write $e_i:=e_i^{(1)}$ and $f_i:=f_i^{(1)}$. We now state some results about these functors.

\begin{lemma}\cite[(8.7)]{K3}
Let $V\in\mod{\F\s_n}$. Then
\[V\da_{\s_{n-1}}\cong \bigoplus_{i\in I}e_iV\quad\text{and}\quad V\ua^{\s_{n+1}}\cong \bigoplus_{i\in I}f_iV.\]
\end{lemma}

\begin{lemma}
\label{lem:div.func}
\cite[Lemma 8.2.2(ii), Theorem 8.3.2]{K3}
The functors $e_i^{(r)}$ and $f_i^{(r)}$ are exact, biadjoint and commute with duality. Moreover, $e_i^{r}\cong (e_i^{(r)})^{\oplus r!}$ and $f_i^{r}\cong (f_i^{(r)})^{\oplus r!}$.
\end{lemma}

\begin{lemma}
\label{lem:shap}
Let $k\in\Z_{\geq 0}$. For $V\in\mod{\F\s_n}$ and $W\in\mod{\F\s_{n-r}}$, we have 
\begin{align*}
\Ext^k_{\s_{n-r}}(e_i^{(r)}V, W)\cong\Ext^k_{\s_{n}}(V,f_i^{(r)}W) 
\quad {\textrm {and}}\quad 
\Ext^k_{\s_{n}}(f_i^{(r)}W, V)\cong\Ext^k_{\s_{n-r}}(W,e_i^{(r)}V).
\end{align*}
\end{lemma}

\begin{proof}
This follows from Lemma~\ref{lem:div.func} and Shapiro's lemma.
\end{proof}

\begin{lemma}
\label{lem:res.irre}
 \cite[Theorems 11.2.10, 11.2.11]{K3}
Let $\la\in \Parreg_n$, $i\in I$ and $r\in\Z_{\geq 1}$. Then: 
\begin{itemize}
\item[(i)] $e_i^{(r)}D^\la\neq 0$ (resp. $f_i^{(r)}D^\la\neq 0$) if and only if $r\leq \eps_i(\la)$ (resp. $r\leq \phi_i(\la)$), in which case $e_i^{(r)}D^\la$ (resp. $f_i^{(r)}D^\la$) is a self-dual indecomposable module with simple socle and head both isomorphic to $D^{ \te_i^r\la}$ (resp. $D^{ \tf_i^r\la})$;
\item[(ii)] $[e_i^{(r)}D^\la:D^{ \te_i^r\la}]={\eps_i(\la)\choose r} 
=\dim\End_{\s_{n-r}}(e_i^{(r)}D^\la)$;
and $[f_i^{(r)}D^\la:D^{ \tf_i^r\la}]={\phi_i(\la)\choose r}=\dim\End_{\s_{n+r}}(f_i^{(r)}D^\la)$;
\item[(iii)] If $D^\mu$ is a composition factor of $e_i^{(r)}D^\la$ (resp. $f_i^{(r)}D^\la$), then $\eps_i(\mu)\leq\eps_i(\la)-r$ (resp. $\phi_i(\mu)\leq\phi_i(\la)-r$), with equality holding if and only if $\mu= \te_i^r\la$ (resp. $\mu= \tf_i^r\la$). In particular, $e_i^{(\eps_i(\la))}D^\la\cong D^{\te_i^{\eps_i(\la)}}$ and 
$f_i^{(\phi_i(\la))}D^\la\cong D^{\tf_i^{\phi_i(\la)}}$. 
\item[(iv)] $e_i^{(r)}D^\la$ (resp $f_i^{(r)}D^\la$) is irreducible if and only if $r=\eps_i(\la)$ (resp. $r=\phi_i(\la)$).
\end{itemize}
\end{lemma}

\begin{lemma}\cite[Lemma 2.24]{GKM}
\label{lem:res.Sp}
Let $\la\in \Par_n$, $i\in I$ and $r\in\Z_{\geq 0}$. 
Then  $e_i^{(r)}S^\la$ (resp. $f_i^{(r)}S^\la$) has a Specht filtration with factors $\{S^{\la_\bA}\mid\bA\in \Om^r(\Rem(\la,i))\}$ (resp. $\{S^{\la^\bB}\mid\bB\in \Om^r(\Add(\la,i))\}$), each appearing once.
\end{lemma}



\section{$p$-separated partitions}\label{separated}

In this section we will prove \cref{thma}. Before proving it, we need the following lemma.

\begin{lemma}\label{L131124}
Let $\la$ be $p$-separated. Fix an abacus configuration $\Ga$ of $\la$. Assume that for some $1\leq j\leq p-1$ we have $r_{j-1}(\Ga)>r_j(\Ga)$ and let $\mu:=\hf_{i_j(\Ga)}^{\phi'_{i_j(\Ga)}(\la)}$. Then:
\begin{enumerate}
\item $\mu$ is $p$-separated,

\item there exists a composition factors of $S^\la$ which has self-extensions if and only if there is a composition factor of $S^\mu$ which has self-extensions.
\end{enumerate}

The same result holds for $j=0$ if $r_{p-1}(\Ga)\geq r_0(\Ga)$.
\end{lemma}

\begin{proof}
The statement for $j=0$ follows from the statement for $j=0$ by taking the abacus configuration of $\la$ with one more bead. So we may assume that $1\leq j\leq p-1$ and $r_{j-1}(\Ga)>r_j(\Ga)$.

Let $i_j:=i_j(\Ga)$.

Since $\la$ is $p$-separated, we have that $h(\la^{(j-1)})+\la^{(j)}_1\leq r_{j-1}(\Ga)-r_{j}(\Ga)$. This means that the last bead on runner $j$ is on a higher row than the first missing bead on runner $j-1$. So runner $j-1$ contains runner $j$ and then $\eps'_{i_j}(\la)=0$. 

Let $\Ga'$ be obtained by exchanging runners $j-1$ and $j$ of $\Ga$. Then $\Ga'$ is an abacus configuration for $\mu$.  By the previous paragraph, runners $j-1$ and $j$ of $\Ga'$ satisfy the $p$-separated condition. For $k\not=\ell$ with neither of $k$ or $\ell$ equal to $j-1$ or $j$, runners $k$ and $\ell$ satisfy the $p$-separated condition for $\Ga'$ since they satisfy it for $\Ga$ and runners $k$ and $\ell$ are unchanged. If $k\not=j-1,j$, then runners $k$ and $j-1$ (resp. $k$ and $j$) satisfy the $p$-separated condition for $\Ga'$ since runners $k$ and $j$ (resp. $k$ and $j-1$) satisfy the $p$-separated condition for $\Ga$. So (i) holds.

Let $D^\nu$ be a composition factor of $S^\la$. As $\la$ has no removable $i_j$-nodes, $e_{i_j}S^\la=0$ by \cref{lem:res.Sp}. In particular $e_{i_j}D^\nu=0$ and so $\eps_{i_j}(\nu)=0$ by \cref{lem:res.irre}.

Let $\Ga(\nu)$ be chosen based on the choice of $\Ga$. Then, by the definition of normal and conormal nodes,
\[\phi_{i_j}(\nu)-\eps_{i_j}(\nu)=\phi'_{i_j}(\nu)-\eps'_{i_j}(\nu)=r_{j}(\Ga(\nu))-r_{j-1}(\Ga(\nu)).\]
Since $\nu$ and $\la$ have the same core we then have that
\[\phi_{i_j}(\nu)-\eps_{i_j}(\nu)=r_{j}(\Ga)-r_{j-1}(\Ga)=\phi'_{i_j}(\la)-\eps'_{i_j}(\la).\]
It then follows that
$\phi_{i_j}(\nu)=\phi'_{i_j}(\la)$.

By \cref{lem:res.irre,lem:res.Sp}, $f_{i_j}^{(\phi'_{i_j}(\la))}S^\la\cong S^\mu$ and $f_{i_j}^{(\phi_{i_j}(\nu))}D^\nu\cong D^{\tf_{i_j}^{\phi_{i_j}(\nu)}\nu}$. Since $\phi_{i_j}(\nu)=\phi'_{i_j}(\la)$ it follows that $D^{\tf_{i_j}^{\phi_{i_j}(\nu)}\nu}$ is a composition factor of $S^\mu$.

Similarly if $D^\pi$ is a composition factor of $S^\mu$ then $\phi_{i_j}(\pi)=0$ and $D^{\te_{i_j}^{\eps_{i_j}(\pi)}\pi}$ is a composition factor of $S^\la$.

So (ii) follows by \cite[Lemma 2.22]{GKM}.
\end{proof}

We are now ready to prove \cref{thma}.

\begin{proof}[Proof of \cref{thma}]
Let $\la$ be $p$-separated. Let the sequence of partitions $\la=\la^0,\ldots,\la^s$ be obtained applying the sequence of moves from the proof of \cite[Lemma 3.1]{F1}. Then $\la^k$ is obtained from $\la^{k-1}$ by swapping two consecutive runners in the corresponding abacus configurations, so that the first of the two runners contains more beads than the other one for $\la^{k-1}$. By repeated application of \cref{L131124}(i), each $\la^k$ is $p$-separated. Further by construction $S^{\la^s}$ is in a RoCK block.

By \cite[Theorem A]{GKM} no composition factor of $S^{\la^s}$ has self-extensions, so the same holds for $S^\la$ by again repeatedly applying \cref{L131124}(ii).
\end{proof}

\section{Reduction results}\label{sreductions}

In this section we will present improvements on reduction results from \cite{GKM}, proving Theorems \ref{thmc} and \ref{thmb}.

We start by looking at Theorem \ref{thmb}. Before proving it, we need the next lemma of independent interest.

\begin{lemma}\label{L081223}
Let $n,a,b\geq 0$ and $i\not=j\in I$. If $\la\in\Parres_{n-a}$ and $\mu\in\Parres_{n-b}$ then
\[\dim\Hom_{\s_n}(f_i^{(a)}D^\la,f_j^{(b)}D^\mu)\leq 1.\]

Similarly if $\la\in\Parres_{n+a}$ and $\mu\in\Parres_{n+b}$ then
\[\dim\Hom_{\s_n}(e_i^{(a)}D^\la,e_j^{(b)}D^\mu)\leq 1.\]
\end{lemma}

\begin{proof}
We start by proving the first statement.

If $a>\phi_i(\la)$ or $b>\phi_j(\mu)$ then $f_i^{(a)}D^\la=0$ or $f_j^{(b)}D^\mu=0$, so the statement clearly holds in these cases. If $a=0$ then 
\[\dim\Hom_{\s_n}(f_i^{(a)}D^\la,f_j^{(b)}D^\mu)=\dim\Hom_{\s_n}(D^\la,f_j^{(b)}D^\mu)\leq 1\]
since $e_j^{(b)}D^\mu$ has simple socle by \cref{lem:res.irre}. Similarly if $b=0$.

So we may assume that $0<a\leq\phi_i(\la)$ and $0<b\leq\phi_j(\mu)$. In these case
\begin{align}\label{E081223_2}
&\dim\Hom_{\s_n}(f_i^{(a)}D^\la,f_j^{(b)}D^\mu)=\frac{1}{a!b!}\dim\Hom_{\s_n}(f_i^aD^\la,f_j^bD^\mu)\\
\nonumber&=\frac{1}{a!b!}\dim\Hom_{\s_{n-b}}(e_j^bf_i^aD^\la,D^\mu)=\frac{1}{a!b!}\dim\Hom_{\s_{n-b}}(f_i^ae_j^bD^\la,D^\mu)\\
\nonumber&=\frac{1}{a!b!}\dim\Hom_{\s_{n-a-b}}(e_j^bD^\la,e_i^aD^\mu)=\dim\Hom_{\s_{n-a-b}}(e_j^{(b)}D^\la,e_i^{(a)}D^\mu),
\end{align}
where the third equality holds by \cite[Lemma 4.7]{M1} and the others by \cref{lem:div.func}.

If $b>\eps_j(\la)$ or $a>\eps_i(\mu)$ then $\dim\Hom_{\s_n}(f_i^{(a)}D^\la,f_j^{(b)}D^\mu)=0$. Otherwise
\begin{align}
\nonumber&\dim\Hom_{\s_n}(f_i^{(a)}D^\la,f_j^{(b)}D^\mu)=\dim\Hom_{\s_{n-a-b}}(e_j^{(b)}D^\la,e_i^{(a)}D^\mu)\\
\nonumber&\leq\dim\Hom_{\s_{n-a-b}}(f_j^{(\eps_j(\la)-b)}D^{\tilde e_j^{\eps_j(\la)}\la},f_i^{(\eps_i(\mu)-a)}D^{\tilde e_i^{\eps_i(\mu)}\mu})\\
\nonumber&=\dim\Hom_{\s_{n-a-b}}(f_i^{(\eps_i(\mu)-a)}D^{\tilde e_i^{\eps_i(\mu)}\mu},f_j^{(\eps_j(\la)-b)}D^{\tilde e_j^{\eps_j(\la)}\la}).
\end{align}
with the inequality holding in view of \cite[Lemma 3.3]{M1} and the last equality by self-duality of the translation functors, due to \cref{lem:div.func} and self-duality of irreducible representations of symmetric groups.

Note that $n-a-b<n$ since $a,b>0$. Repeating the above argument possibly multiple times, we obtain that
\[\dim\Hom_{\s_n}(f_i^{(a)}D^\la,f_j^{(b)}D^\mu)\leq\dim\Hom_{\s_m}(f_i^{(c)}D^\al,f_j^{(d)}D^\be)\leq 1\]
for some $\al,\be$ and some $c,d$ such that at least one of $c=0$, $c>\phi_i(\al)$, $d=0$ or $d>\phi_j(\be)$ holds or that
\[\dim\Hom_{\s_n}(f_i^{(a)}D^\la,f_j^{(b)}D^\mu)\leq\dim\Hom_{\s_m}(e_j^{(c)}D^\al,e_i^{(d)}D^\be)=0\]
for some $\al,\be$ and some $c,d$ with $c>\eps_j(\al)$ or $d>\eps_i(\be)$.
 
The second statement follows from the first and looking at the first and last term of \eqref{E081223_2}.
\end{proof}

We are now ready to prove Theorem \ref{thmb}. We actually prove the following stronger version of Theorem \ref{thmb}.

\begin{thm}\label{T081223}
Let $\la\in\Par_p(n)$. Assume that for some $i\not=j$,
\[\Ext^1_{\s_{n-\eps_i(\la)}}(D^{\tilde e_i^{\eps_i(\la)}\la},D^{\tilde e_i^{\eps_i(\la)}\la})=0=\Ext^1_{\s_{n-\eps_j(\la)}}(D^{\tilde e_j^{\eps_j(\la)}\la},D^{\tilde e_j^{\eps_j(\la)}\la}).\]
Then $\Ext^1_{\s_n}(D^\la,D^\la)=0$.

In particular if $\mu\in\Par_p(m)$ is minimal with $\Ext^1_{\s_m}(D^\mu,D^\mu)\not=0$ then all normal nodes of $\mu$ have the same residue.
\end{thm}

\begin{proof}
The second statement follows from the first. So we may assume that $\Ext^1_{\s_n}(D^\la,D^\la)\not=0$ but 
\[\Ext^1_{\s_{n-\eps_i(\la)}}(D^{\tilde e_i^{\eps_i(\la)}\la},D^{\tilde e_i^{\eps_i(\la)}\la})=0=\Ext^1_{\s_{n-\eps_j(\la)}}(D^{\tilde e_j^{\eps_j(\la)}\la},D^{\tilde e_j^{\eps_j(\la)}\la}).\]
From the short exact sequence
\[0\to D^\la\to f_i^{(\eps_i(\la))}D^{\tilde e_i^{\eps_i(\la)}\la}\to f_i^{(\eps_i(\la))}D^{\tilde e_i^{\eps_i(\la)}\la}/D^\la\to 0,\]
the corresponding long exact sequence for $\Ext^i(D^\la,-)$ and \cref{lem:shap,lem:res.irre} we obtain that
\begin{equation}\label{E131124}
\Ext^1_{\s_n}(D^\la,D^\la)\cong\Hom_{\s_n}(D^\la,f_i^{(\eps_i(\la))}D^{\tilde e_i^{\eps_i(\la)}\la}/D^\la).
\end{equation}

Let $V\cong D^\la|D^\la$ (such a $V$ exists since $\Ext^1_{\s_n}(D^\la,D^\la)\not=0$). Since $\soc(f_i^{(\eps_i(\la))}D^{\tilde e_i^{\eps_i(\la)}\la})\cong D^\la$ by \cref{lem:res.irre}, we have by (\ref{E131124}) that $V\subseteq f_i^{(\eps_i(\la))}D^{\tilde e_i^{\eps_i(\la)}\la}$. Similarly $V^*\subseteq f_j^{(\eps_i(\la))}D^{\tilde e_j^{\eps_j(\la)}\la}$, so that $V$ is a quotient of $f_j^{(\eps_i(\la))}D^{\tilde e_j^{\eps_j(\la)}\la}$.

It then follows that
\[\dim\Hom_{\s_n}(f_j^{(\eps_i(\la))}D^{\tilde e_j^{\eps_j(\la)}\la},f_i^{(\eps_i(\la))}D^{\tilde e_i^{\eps_i(\la)}\la})\geq\End_{\s_n}(V)=2,\]
contradicting \cref{L081223}.
\end{proof}

We will next prove Theorem \ref{thmc}. We start with the following lemma:

\begin{lemma}\label{L081223_2}
Let $\la\in\Par_p(n)$ and $i\in I$. Assume that $V\cong D^\la|D^\la\subseteq f_i^{(\eps_i(\la))}D^{\tilde e_i^{\eps_i(\la)}\la}$ and let $b,c$ minimal with $V\subseteq f_i^{(b)}D^{\tilde e_i^{b}\la}$ and with $V$ a quotient of $e_i^{(c)}D^{\tilde f_i^{c}\la}$. Then both $b$ and $c$ exist and $b=c\geq 1$.
\end{lemma}

\begin{proof}
Clearly $b$ exists. By \cite[Lemma 3.4]{M1}, $f_i^{(\eps_i(\la))}D^{\tilde e_i^{\eps_i(\la)}\la}\cong e_i^{(\phi_i(\la))}D^{\tilde f_i^{\phi_i(\la)}\la}$. Since by \cref{lem:res.irre}
\[\dim\End_{\s_{n-r}}(e_i^{(\phi_i(\la))}D^{\tilde f_i^{\phi_i(\la)}\la})=[e_i^{(\phi_i(\la))}D^{\tilde f_i^{\phi_i(\la)}\la}:D^\la]\]
and $e_i^{(\phi_i(\la))}D^{\tilde f_i^{\phi_i(\la)}\la}$ has head and socle isomorphic to $D^\la$, it follows that $V$ is also a quotient of $e_i^{(\phi_i(\la))}D^{\tilde f_i^{\phi_i(\la)}\la}$. So $c$ also exists.

Since $f_i^{(0)}D^{\tilde e_i^0\la}\cong D^\la\cong e_i^{(0)}D^{\tilde f_i^0\la}$, we have that $b,c\geq 1$. By minimality of $b$ and $c$ we also have that $V\not\subseteq f_i^{(b-1)}D^{\tilde e_i^{b-1}\la}$ and that $V$ is not a quotient of $e_i^{(c-1)}D^{\tilde f_i^{c-1}\la}$.

We will only prove that $c\leq b$ since $b\leq c$ can be proved similarly.

Since $e_i^{(c)}D^{\tilde f_i^c\la}$ and $e_i^{(c-1)}D^{\tilde f_i^{c-1}\la}$ have simple head isomorphic to $D^\la$ by \cref{lem:res.irre} and since $e_i^{(a)}$ and $f_i^{(a)}$ are biadjoint functors for every $a$ by \cref{lem:div.func} we have that
\begin{align*}
\dim\Hom_{\s_{n+c}}(D^{\tilde f_i^c\la},f_i^{(c)}V)&=\dim\Hom_{\s_n}(e_i^{(c)}D^{\tilde f_i^c\la},V)=2,\\
\dim\Hom_{\s_{n+c-1}}(D^{\tilde f_i^{c-1}\la},f_i^{(c-1)}V)&=\dim\Hom_{\s_n}(e_i^{(c-1)}D^{\tilde f_i^{c-1}\la},V)=1.
\end{align*}

In view of \cref{lem:div.func}
\[f_i^{(c-1)}V\subseteq f_i^{(c-1)}f_i^{(b)}D^{\tilde e_i^{b}\la}\subseteq (f_i^{(b+c-1)}D^{\tilde e_i^{b}\la})^{\oplus\binom{b+c-1}{b}}.\]
By \cref{lem:res.irre}, $\soc(f_i^{(b+c-1)}D^{\tilde e_i^{b}\la})\cong D^{\tilde f_i^{c-1}\la}$. It then follows that $\soc(f_i^{(c-1)}V)\cong D^{\tilde f_i^{c-1}\la}$ and then that $f_i^{(c-1)}V\subseteq f_i^{(b+c-1)}D^{\tilde e_i^{b}\la}$. So
\begin{align*}
2c&=c\dim\Hom_{\s_{n+c}}(D^{\tilde f_i^c\la},f_i^{(c)}V)=\dim\Hom_{\s_{n+c}}(D^{\tilde f_i^c\la},f_if_i^{(c-1)}V)\\
&\leq \dim\Hom_{\s_{n+c}}(D^{\tilde f_i^c\la},f_if_i^{(b+c-1)}D^{\tilde e_i^{b}\la})\\
&=(b+c)\dim\Hom_{\s_{n+c}}(D^{\tilde f_i^c\la},f_i^{(b+c)}D^{\tilde e_i^{b}\la})=b+c,
\end{align*}
with the last equality on the first row on the first equality on the last row holding by \cref{lem:div.func} and the last equality holding by \cref{lem:res.irre}. In particular $c\leq b$.
\end{proof}

We are now ready to prove Theorem \ref{thmc}.

\begin{proof}[Proof of Theorem \ref{thmc}]
By \cref{L081223_2}, $b\geq 1$ and $b$ is also minimal such that $V$ is a quotient of $e_i^{(b)}D^{\tilde f_i^{b}\la}$. In particular $f_i^{(b)}D^{\tilde e_i^{b}\la},e_i^{(b)}D^{\tilde f_i^{b}\la}\not=0$, so $1\leq b\leq\min\{\eps_i(\la),\phi_i(\la)\}$.
 
By \cref{lem:div.func} and the proof of \cref{L081223_2} (and using that $b=c$),
\begin{align*}
\dim\Hom_{\s_{n+b-1}}(e_iD^{\tilde f_i^b\la},f_i^{(b-1)}V)=\dim\Hom_{\s_{n+b}}(D^{\tilde f_i^b\la},f_if_i^{(b-1)}V)=2b.
\end{align*}

By the proof of \cref{L081223_2} we also have that $f_i^{(b-1)}V\subseteq f_i^{(2b-1)}D^{\tilde e_i^{b}\la}$. In view of \cref{lem:res.irre} and \cite[Lemma 2.1]{GKM}, no two distinct submodules of $f_i^{(2b-1)}D^{\tilde e_i^{b}\la}$ are isomorhic. Then same then hold also for submodules of $f_i^{(b-1)}V$. 

By \cite[Lemma 3.9]{KMT}
\[e_iD^{\tilde f_i^b\la}=W_{\eps_i(\la)+b}\supseteq W_{\eps_i(\la)+b-1}\supseteq\ldots\supseteq W_1\cong D^{\tilde f_i^{b-1}\la}\]
with $W_j$ self-dual, with head and socle isomorphic to $D^{\tilde f_i^{b-1}\la}$ and with $[W_j:D^{\tilde f_i^{b-1}\la}]=j$ for each $1\leq j\leq\eps_i(\la)+b$. By \cite[Lemma 2.6]{GKM} if $W$ is any submodule of $e_iD^{\tilde f_i^b\la}$ with only $D^{\tilde f_i^{b-1}\la}$ appearing in the head of $W$ then $W=W_j$ for some $j$.


Since $e_iD^{\tilde f_i^b\la}$ is self-dual by \cref{lem:res.irre}, there similarly exist submodules
\[M_{\eps_i(\la)+b}\subseteq M_{\eps_i(\la)+b-1}\subseteq\ldots\subseteq M_1\subseteq e_iD^{\tilde f_i^b\la}\]
such that $e_iD^{\tilde f_i^b\la}/M_K\cong W_K$ for every $k$. Further if $M\subseteq e_iD^{\tilde f_i^b\la}$ with $e_iD^{\tilde f_i^b\la}/M$ having only $D^{\tilde f_i^{b-1}\la}$ appearing in its socle, then $M=M_k$ for some $k$. Further distinct quotients of $e_iD^{\tilde f_i^b\la}$ are pairwise non-isomorphic.


Since $\dim\Hom_{\s_{n+b-1}}(e_iD^{\tilde f_i^b\la},f_i^{(b-1)}V)=2b$, it follows by the last three paragraphs that there exists $2b$ distinct $k$ such that $W_k$ is a submodule of $f_i^{(b-1)}V$.

In particular there exists $k\geq 2b$ with $W_k\subseteq f_i^{(b-1)}V$. As $W_a\subseteq W_k$ for every $a\leq k$, we in particular have that $W_{b+j}\subseteq W_k\subseteq f_i^{(b-1)}V$ for every $1\leq j\leq b$.

Assume that $\la_{B_1,\ldots,B_b}^{C_j}$ is $p$-regular, so that $\soc((S^{\la_{B_1,\ldots,B_b}^{C_j}})^*)\cong D^{\la_{B_1,\ldots,B_b}^{C_j}}$. Then by \cite[Lemma 2.6]{GKM} and \cite[Lemma 3.4]{M2} $D^{\la_{B_1,\ldots,B_b}^{C_j}}$ is a composition factor of $W_{b+j}$ and then also of $f_i^{(b-1)}V$. Since $V\cong D^\la|D^\la$, it follows that $D^{\la_{B_1,\ldots,B_b}^{C_j}}$ is a composition factor of $f_i^{(b-1)}D^\la$. By \cite[Corollary 2.26]{GKM}, $\la_{B_1,\ldots,B_b}^{C_j}\unrhd\la_{D_1,\ldots,D_{b-1}}$ for some $i$-removable nodes $D_1,\ldots,D_{b-1}$ of $\la$.

Similarly, starting from
\begin{align*}
\dim\Hom_{\s_{n-b+1}}(f_iD^{\tilde e_i^b\la},e_i^{(b-1)}V^*)=\dim\Hom_{\s_{n-b+1}}(e_i^{(b-1)}V,f_iD^{\tilde e_i^b\la})=2b
\end{align*}
and using \cite[Lemma 3.11]{KMT} instead of \cite[Lemma 3.9]{KMT}, if $\la_{B_j}^{C_1,\ldots,C_b}$ is $p$-regular and $1\leq j\leq b$, then $\la_{B_j}^{C_1,\ldots,C_b}\unrhd\la^{E_1,\ldots,E_{b-1}}$ for some $i$-addable nodes $E_1,\ldots,E_{b-1}$ of $\la$.

Let $B_1$ be on row $r$. Then $B_j$ is in row $\leq r$ for every $j$, while $C_j$ is on row $>r$ for every $j$. Then
\[\sum_{k\geq r+1}(\la_{B_1,\ldots,B_b}^{C_j})_k=\sum_{k\geq r+1}\la_k+1\geq \sum_{k\geq r+1}(\la_{D_1,\ldots,D_{b-1}})_k+1\]
and
\[\sum_{k\leq r}(\la_{B_j}^{C_1,\ldots,C_b})_k=\sum_{k\leq r}\la_k-1\geq \sum_{k\leq r}(\la^{E_1,\ldots,E_{b-1}})_k-1\]
contradicting $\la_{B_1,\ldots,B_b}^{C_j}\unrhd\la_{D_1,\ldots,D_{b-1}}$ or $\la_{B_j}^{C_1,\ldots,C_b}\unrhd\la^{E_1,\ldots,E_{b-1}}$.
\end{proof}

\begin{cor}\label{C081223}
Let $\la\in\Par_p(n)$ and $V\cong D^\la|D^\la$. Assume that for some $b\geq 1$, $V\subseteq f_i^{(b)}D^{\tilde e_i^{b}\la}$ but $V\not\subseteq f_i^{(b-1)}D^{\tilde e_i^{b-1}\la}$. Then for some $h\geq 0$ and $a\geq b$, $\la$ is of the form
\begin{align*}
&(\la_1,\ldots,\la_h,a+b,(a+b-1)^{p-1},\ldots,(a+1)^{p-1},a^{p-2},(a-1)^{p-1},\ldots,(a-b+1)^{p-1},\\
&a-b,\la_{h+(2b-1)(p-1)+2},\la_{h+(2b-1)(p-1)+3},\ldots),
\end{align*}
the bottom $b$ $i$-conormal nodes of $\la$ are the nodes
\[\{(h+(b-1-j)(p-1)+1,a+1+j)|0\leq j\leq b-1\}\]
and the top $b$ $i$-conormal nodes of $\la$ are the nodes
\[\{(h+(b-1+j)(p-1),a+1-j)|1\leq j\leq b\}.\]
\end{cor}

\begin{proof}
This follows from \cref{T081223_2} through a simple combinatorial argument, since $\la_{B_1,\ldots,B_a}$ and $\la^{C_1,\ldots,C_a}$ are $p$-regular but $\la_{B_1,\ldots,B_a}^{C_j}$ and $\la_{B_j}^{C_1,\ldots,C_a}$ are $p$-singular for every $1\leq j\leq a$.
\end{proof}

\end{document}